\documentclass[reqno]{amsart}
\usepackage{amssymb}
\usepackage[dvips]{epsfig}
\usepackage{graphicx}
\usepackage{color}
\usepackage{amsmath}
\usepackage{amssymb}
\usepackage{amsfonts}
\usepackage{amsthm}
\usepackage{bbm}
\usepackage{tikz}

\newcommand{\R}{\mathbb R}

\newcommand{\Z}{\mathbb Z}

\newcommand{\C}{\mathbb C}

\newcommand{\N}{\mathbb{N}}
\newcommand{\T}{\mathbb{T}}

\newtheorem{thm}{Theorem}[section]
\newtheorem{lem}[thm]{Lemma}
\newtheorem{prop}[thm]{Proposition}

\theoremstyle{remark}
\newtheorem{rem}{\bf Remark}[section]
\theoremstyle{definition}
\newtheorem{defn}[thm]{Definition}

\numberwithin{equation}{section}
\usepackage[colorlinks=true,linkcolor=blue,citecolor=blue]{hyperref}

\begin{document}

\title[]{Absence of eigenvalues of analytic quasi-periodic Schr\"odinger operators on $\R^d$}
\author{Yunfeng Shi}
\address[Y. Shi] {College of Mathematics,
Sichuan University,
Chengdu 610064,
China}

\date{\today}


\begin{abstract}
In this paper  we study  on $L^2(\R^d)$ the quasi-periodic Schr\"odinger operator $H=-\Delta+ \lambda V(x),$ where $V$ is a real analytic quasi-periodic function and $\lambda>0$. We first show that $H$ has no eigenvalues in \textit{low energy region}. We also provide in \textit{low energy region} the new phase transition parameter, i.e. the competition between the strength of coupling and the length for frequencies. 

\end{abstract}

\maketitle






\maketitle
\section{Introduction and main results}

Let $H=-\Delta+V(x)$ be the Schr\"odinger operator defined on $L^2(\R^d)$. The question of determining for which potential $V(x)$ such $H$  has no positive eigenvalues  attracted  a great deal of attention  over years, see e.g., \cite{Kat59,Sim69, Agm70,FH83,JK85,IJ03,Mar19}. In general,  those works require $V(x)$ to obey certain decaying law at $\infty$.   This may exclude a large class of potentials without  decaying, such as  the quasi-periodic one. The aim of the present paper  tries to establish absence of positive eigenvalues for some quasi-periodic Schr\"odinger operators on $\R^d$ for \textit{arbitrary}  $d\geq 1$.

Indeed, the quasi-periodic Schr\"odinger operator  has important applications in condensed matter physics, and has been extensively studied in recent years, in particular in its discrete case (see \cite{MJ17,Dam17} and references therein).

Considerable research results have been obtained  for  quasi-periodic Schr\"odinger operators on $\R$. In $1D$  and periodic potentials case, the Floquet theory works well,  and as a result the spectrum consists of intervals (bands) and is purely absolutely continuous (\textit{ac}).  When the periodic potential is replaced by a quasi-periodic one, the spectral properties change dramatically resulting from the so-called \textit{small-divisors} effect. Dinaburg-Sinai \cite{DS75} first proved the existence of \textit{ac}  spectrum by a using a KAM reducibility argument. This result was further developed later by R\"ussmann \cite{Rus80} and Moser-P\"oschel \cite{MP84}. Surace \cite{Sur90} showed for  potentials of the form
\begin{align}\label{surp}
V(x)=\lambda(\cos x+\cos(\theta+\omega x)),
\end{align}
the whole spectrum contains no eigenvalues if $\lambda$ is small and the frequency $\omega$ is Diophantine. He extended  the multi-scale analysis (MSA) of \cite{FS83,FSW90} to work in momentum space. The breakthrough came from Eliasson \cite{Eli92}, in which he proved for  any nonconstant analytic quasi-periodic potential and Diophantine frequency, the spectrum is purely \textit{ac} in high energy region.  His result also applied to the whole spectrum for small quasi-periodic potentials. Moreover, he  showed  the spectrum is a Cantor set for a generic set of analytic quasi-periodic potentials, which exhibits completely different spectral features compared with periodic potentials. For quasi-periodic potentials $\lambda V(x)$ with large $\lambda>0$, Anderson localization  (AL, i.e., pure point spectrum with exponentially decaying eigenfunctions) and de-localization transition \cite{YZ14} may be expected.  Sorets-Spencer \cite{SS91} improved Herman's subharmonic trick to obtain the positivity of Lyapunov exponent at low energy  for potentials \eqref{surp} if $\lambda\gg1$.  Fr\"ohlich-Spencer-Wittwer \cite{FSW90} proved the first AL for potentials of the form \eqref{surp} at low energy if $\lambda\gg1$. They performed a MSA for Green's functions  in  $x\in\R$ space directly. It has been proven by You-Zhou \cite{YZ14} that
there exists the phase transition from singular to purely \textit{ac} spectrum for  $1D$ and two frequencies quasi-periodic operators if the coupling is large.  Very recently, Binder-Kinzebulatov-Voda \cite{BKV17} proved a non-perturbative AL (i.e. AL under only positive Lyapunov exponents assumption) for general analytic potentials in \textit{finite} energy intervals by applying  methods of \cite{BG00,GS08}. For more results in $1D$ quasi-periodic setting, we refer to \cite{Bje06,DG14,Liu18}.

If one increases the space dimension to $d\geq2$, the situation becomes significantly more complicated and  much less results were obtained in this setting.  Unlike the $1D$  case,  it was conjectured  (i.e. the Bethe-Sommerfeld conjecture \cite{Par08,PS10}) that the spectrum of any periodic Schr\"odinger operator in higher dimensions  contains finitely many gaps.
By contrast, Damanik-Fillman-Gorodetski \cite{DFG19} had constructed $dD$ almost-periodic Schr\"odinger operators whose spectrum  is a generalized Cantor set of zero Lebesgue measure. The first  result concerning  \textit{ac} spectrum of a $2D$ quasi-periodic Schr\"odinger operator was due to Karpeshina-Shterenberg \cite{KS19}, where they obtained  the existence of \textit{ac} spectrum at high energy,  and the spectrum contains a semi-axis. The proof of this  elegant result is based on  their new  MSA in momentum space. Before this work, they \cite{KS13} got a similar result for $(-\Delta)^l+V(x)$ with $l\geq 2$. We also mention the work of Karpeshina-Lee-Shterenberg-Stolz \cite{KLSS17},  in which the existence of ballistic transport for the Schr\"odinger operator with limit-periodic or quasi-periodic potential in dimension two was established. There are also some results \cite{PS16,PS12} for quasi-periodic (even almost periodic) operators on $\R^d$ in dealing with  asymptotic expansion of the spectral functions (such as the IDS).

If $d\geq 3$, as far as we know, there are no localization or de-localization results available for a quasi-periodic Schr\"odinger operator on $\R^d$. 
This is the other  motivation of the present work.

In this paper we study on $\R^d$ ($d\geq 1$) the Schr\"odinger operators with analytic quasi-periodic potentials $\lambda V(x)$ and prove the absence of eigenvalues in \textit{ low energy region}. Combined our main theorem and results of \cite{Bje06,Bje07,YZ14,BKV17}, we also provide the new phase transition parameter in \textit{low energy region} (see (4) of Remark \ref{rm1} in the following).

Here is the set up of our main result:

Let $b_i\in\N$ ($1\leq i\leq d$) and  $b=\sum_{i=1}^db_i>d$. Define  $\T^{b}=\prod_{i=1}^d(\R/2\pi\Z)^{b_i}.$
Write $\vec\theta=(\vec\theta_1,\cdots,\vec\theta_d)\in \T^{b}$ for $\vec\theta_i=(\theta_{i1},\cdots,\theta_{ib_i})\in\T^{b_i}$ ($1\leq i\leq d$).  Similarly,  let $\vec\omega=(\vec{\omega}_1,\cdots,\vec\omega_d)\in[0,2\pi]^{b}$ with $\vec\omega_i\in[0,2\pi]^{b_i}$ ($1\leq i\leq d$). Let $ x=(x_1,\cdots,x_d)\in\R^d$ and define $x\vec{\omega}=(x_1\vec\omega_1,\cdots,x_d\vec\omega_d)\in\R^{b}.$

We consider on $L^2(\R^d)$ the following  Schr\"odinger operators with  quasi-periodic potentials
\begin{equation}\label{qps}
H(\vec\theta)=-\Delta+ \lambda V(\vec\theta+Kx\vec{\omega}),
\end{equation}
where the real function $V$  is the potential, $\lambda\geq 0$ is the coupling,  $\vec\omega$ {is} the frequency,  $\vec\theta$ is the phase and $K>0$.

We focus on the analytic function $V\in C^\omega(\T^{b};\R)$ satisfying $\int_{\T^{b}}V(\vec\theta){\rm d}\vec\theta=0$. Without loss of generality, we may assume for some $\rho>0$ and $\forall\ \mathbf{k}=(\mathbf{k}_1,\cdots,\mathbf{k}_d)\in\Z^{b}$ ($\mathbf{k}_i=(k_{i1},\cdots,k_{ib_i})\in\Z^{b_i}, 1\leq i\leq d$)
\begin{align}\label{gvc}
|\widehat V_\mathbf{k}|\leq  e^{-\rho|\mathbf{k}|},
\end{align}
where $\widehat V_\mathbf{k}=\int_{\T^{b}}V(\vec\theta)e^{-\mathbf{i}\mathbf{k}\cdot\vec\theta}{\rm d}\vec\theta$ (with $\mathbf{k}\cdot\vec\theta=\sum\limits_{i=1}^d\mathbf{k}_i\cdot\vec\theta_i=\sum\limits_{i=1}^d\sum\limits_{j=1}^{b_i}{k}_{ij}\theta_{ij} $) is the Fourier coefficient of $V(\vec\theta)$ and $|\mathbf{k}|=\max\limits_{1\leq i\leq d, 1\leq j\leq b_i}|{k}_{ij}|$.


Denote by ${\rm mes}(\cdot)$ the Lebesgue measure. We have 
\begin{thm}\label{mthm1}
Let $H(\vec\theta)$ be defined by \eqref{qps} with $V$ satisfying \eqref{gvc}. Then there is $c_\star=c_\star(b,d)>0$ (depending on $b_1,\cdots,b_d,d$) such that, for $\delta>0$ and $0<c_1<c_\star$, there exists $\varepsilon_0=\varepsilon_0(\delta,c_1,\rho,b,d)>0$ such that the following holds: If $0<\frac{\lambda}{K^2}\leq \varepsilon_0$, then there is some $\Omega=\Omega(K,\lambda)\subset[0,2\pi]^{b}$ satisfying ${\rm mes}([0,2\pi]^{b}\setminus \Omega)\leq \delta$ such that $H(\vec\theta)$ has no eigenvalues in $\left[-K^2(\log\frac{K^2}{\lambda})^{\frac{1}{2c_1}}, K^2(\log\frac{K^2}{\lambda})^{\frac{1}{2c_1}}\right]$ for every $\vec\theta \in \T^{b}$ and $\vec\omega\in\Omega$.
\end{thm}

\begin{rem}\label{rm1}
\begin{itemize}
\item[(1)] Theorem \ref{mthm1} is non-vacuous, i.e., there exists a  portion of the spectrum in $\mathcal{I}=\left[-K^2(\log\frac{K^2}{\lambda})^{\frac{1}{2c_1}}, K^2(\log\frac{K^2}{\lambda})^{\frac{1}{2c_1}}\right]$. Indeed (see Lemma \ref{spt} in the Appendix), one can show for any $E\geq 0$ and $M\geq\lambda |V|_{\max}$, $\sigma(H(\vec\theta))\cap [E-M,E+M]\neq \emptyset,$
where $\sigma(\cdot)$ denotes the spectrum. In particular, $\sigma(H(\vec\theta))\cap [-M, M]\neq \emptyset$ for any $M>\lambda |V|_{\max}$. Moreover, it is obvious that $\mathcal{I}\nearrow \R$ as $\lambda\to 0$ or $K\to\infty$.

\item[(2)] We make no \textit{smallness} assumptions on $\lambda>0$,  and prove the first \textit{purely continuous} spectrum result in \textit{low energy region} for quasi-periodic operators on $\R^d$ for \textit{arbitrary} $d\geq 1$. This can  be regarded as a supplement of results of \cite{KS19}, where the existence of \textit{ac} spectrum for a $2D$ quasi-periodic Schr\"odinger operator was established in \textit{high energy region}.
\item[(3)]
The theorem suggests that the length of frequencies is relevant for the spectral properties of continuous quasi-periodic Schr\"odinger operators.
\item[(4)] 
If $d=1, b_1>1$ and $V$ satisfies some conditions,  it was proved in \cite{BKV17} (combined with results of \cite{Bje06}) that  for $\lambda\geq \lambda_0(K\vec\omega,b_1)>0$ the spectrum in $[0, \lambda^{2/3}]$ is pure point with exponential decay eigenfunctions (i.e. AL) for a.e. Diophantine $\vec\omega$. By contrast, our Theorem \ref{mthm1} implies that  there are no eigenvalues in $[0, C\lambda]$ at all for any $\lambda>0$ if $K\geq C\sqrt{\lambda}$. Those motivate us to provide the new phase transition parameter $\frac{\log K}{\log \lambda}$ from \textit{pure point} to \textit{purely continuous} spectrum in \textit{low energy region}. Precisely,  there may exist a competition between $\lambda$ and $K$: $\frac{\log K}{\log \lambda}>C>0$ implies \textit{purely continuous} spectrum; and $\frac{\log K}{\log \lambda}\leq C$ may show \textit{pure point} spectrum (even AL). We refer to \cite{YZ14} for the phase transition parameter involving the coupling and energy: $\frac{\log E}{\log \lambda}$.
\item[(5)] Our results can be extended to the case $H=(-\Delta)^l+\lambda V(\vec\theta+K{x\vec\omega})$, where $l\in \Z, l\geq 1$ and $V(\cdot)$ is Gevrey regular (see \cite{ShiG}).
\item[(6)] It should be true that the spectrum in $\mathcal{I}$ is purely absolutely continuous for $0<\frac{\lambda}{K^2}\ll1$ and for most $\vec\omega$. However, we think  Anderson localization occurs at low energy   for  $\frac{\lambda}{K^2}\gg1$  and for most $\vec\omega$.
\end{itemize}
\end{rem}


 We outline the proof of Theorem \ref{mthm1}:

 The main scheme is to perform  MSA in momentum space $\Theta\in\R^d$.
 Let $E$ be an eigenvalue of $ H(\vec\theta)$, i.e. $ H(\vec\theta)\Psi=E\Psi$ for some $0\neq\Psi\in L^2(\R^d)$.
Let
 \begin{align*}
 \mathbf{k}\vec\omega=(\mathbf{k}_1\cdot\vec\omega_1,\cdots,\mathbf{k}_d\cdot\vec\omega_d).
 \end{align*}
Denote by $\widehat\Psi$ the Fourier transformation of $\Psi$. Then by direct computations (see Lemma \ref{shnol} for details), the vector $$Z=\left\{Z_\mathbf{k}=e^{-\mathbf{i}\mathbf{k}\cdot\vec\theta}\widehat\Psi(\Theta'+K{\mathbf{k}\vec\omega})\right\}_{\mathbf{k}\in\Z^{b}}$$ will satisfy $h(\Theta)Z=\frac{E}{K^2}Z$, where
 \begin{align*}
(h(\Theta)Z)_\mathbf{k}=\sum_{\mathbf{k}'\in\Z^{b}}\varepsilon \widehat{V}_{\mathbf{k}-\mathbf{k}'}Z_{\mathbf{k}'}+\sum_{i=1}^d(\Theta_i+\mathbf{k}_i\cdot\vec\omega_i)^2Z_\mathbf{k}, \varepsilon=\frac{\lambda}{K^2}, \Theta=\frac{\Theta'}{K}.
 \end{align*}
 In addition, we can even show $Z_\mathbf{k}$ grows at most polynomially in $\mathbf{k}$ for a.e. $\Theta$. Thus the asymptotic property of  $Z_\mathbf{k}$ can be controlled  by \textit{ fast} off-diagonal decaying  properties of the  Green's functions (if exist)
 \begin{align*}
 G_\Lambda(E;\Theta)=\left(R_\Lambda (h(\Theta)-E)R_\Lambda\right)^{-1}, \Lambda\subset\Z^{b}.
 \end{align*}

 At this stage, it suffices to get certain off-diagonal decaying estimates on $G_{\Lambda}(E;\Theta)$  as $|\Lambda|\to\infty$. For this purpose, it needs to make quantitative restrictions on $\vec\omega$,  $\Theta$  and even on $E$  due to the {small-divisors difficulty}. Usually, the KAM and MSA  methods are  powerful  to overcome the small-divisors.   In continuous quasi-periodic operators case, Surace \cite{Sur90} performed the first MSA scheme \footnote{The main scheme is similar that of \cite{FSW90}.} in momentum space $\Theta\in \R$ but restricted to diagonal part  of the form
 \begin{align*}
(\Theta+k_1+k_2\omega)^2,\  (k_1,k_2)\in\Z^2.
 \end{align*}
 Because of this special structure, Surace can even obtain uniform in $E\in\R$ estimate. However, the proof of \cite{Sur90} depends heavily on eigenvalue variations and is hard to work for $\Theta$ in higher dimensions.   For recent $\Theta\in\R$ type MSA, we also refer to \cite{DG14} in dealing with the inverse spectral theory for quasi-periodic operators on $\R$.

 For general $\Theta\in \R^d$ in momentum space the only known  MSA is developed in recent  papers \cite{KS13,KS19} for $d=2$. However, for \textit{arbitrary} $d\geq1$ similar type of questions has been extensively studied in case $\Theta\in\T^d$  \cite{BGS02,Bou07,BK19,JLS20,ShiG} via  techniques of semi-algebraic geometry arguments and subharmonic estimates. As compared with compact momentum case (i.e. $\Theta\in\T^d$), there comes new difficulty for general $\Theta\in\R^d$:  For any finite $\Lambda\subset\Z^{b}$,
 \begin{align}\label{smd12}
\sup_{\mathbf{k}\in \Lambda}\left|\sum_{i=1}^d(\Theta_i+\mathbf{k}_i\cdot\vec\omega_i)^2-E\right|\ll1
 \end{align}
 may frequently  occur  while $|\Theta|\gg1$ and  $|E|\gg1$. To avoid this difficulty, one may make the assumption that for some $C>0$ \footnote{This assumption is 	essentially satisfied in  \cite{BGS02,Bou07,BK19,JLS20,ShiG}.},
\begin{align}\label{bde}
-C\leq E\leq C.
\end{align}
Once \eqref{bde} is satisfied,  $\Theta$ stays essentially in a compact set of size $\sim|\Lambda|$, and the small divisors presented in \eqref{smd12} could be handled via methods of \cite{Bou07,JLS20}. The elimination  frequencies argument of Bourgain \cite{Bou07} based on semi-algebraic geometry theory is likely applicable. Consequently, we may expect that $ H(\vec\theta)$ has no eigenvalues in interval $[-CK^2, CK^2]$ if $0<\varepsilon\leq \varepsilon_0(C)$.

It is reasonable  that the absence of eigenvalues for $ H(\vec\theta)$ should hold in a \textit{much longer} energy interval,  which of  course needs a proof with new ideas. In this paper we will provide such a proof based on the  following observations:
\begin{itemize}
\item The first step of MSA is in general based on perturbative argument,  in particular the Neumann series expansions. It is important that the large deviation estimate (LDE) for Green's functions is valid for \textit{every} $E\in\R$ in this step.
\item Once the first step of MSA is finished, we will deal with the second iteration step. To propagate the LDE, some restrictions on $\vec\omega$ are necessary, which is essentially necessary even in $\Theta\in\T^d$ case. What's new here is that we can  make \textit{quantitative} restrictions on $E$ in this step. More precisely, the first step MSA implies that LDE holds in scales interval $N_0\leq N\leq |\log\varepsilon|^C$. The second MSA step will start with $N=|\log \varepsilon|^C$. To avoid \eqref{smd12}, we can actually set $|E|\leq |\log \varepsilon|^C$.  In this energy interval, we could perform the iteration similar to that of \cite{Bou07,JLS20} together with some technical improvements.
\item Since $|E|\leq |\log \varepsilon|^C$ is negligible as compared with the later general MSA scales $N\gg |\log \varepsilon|^C$,  the iterations  become similar to that done in the second step.
\end{itemize}
In conclusion, our new aspect here is that  we focus on the MSA (in the momentum space $\Theta\in\R^d$) in the first two  steps and make effective restrictions on $E$ in the second step.

Of course, the final aim is to show  $ H(\vec\theta)$ has no eigenvalues on $\R$, which seems difficult to handle via the present method.

The structure of the paper is as follows. Some preliminaries are introduced in \S2. The MSA in momentum space is established in \S3. In \S4, we finish the proof of Theorem \ref{mthm1}.  Some useful estimates are included in the Appendix.

\section{Preliminaries}

\subsection{Some notation}

For any $x\in\mathbb{R}^d$, let $|x|=\max\limits_{1\leq i\leq d}|x_i|$. For $\Lambda\subset\mathbb{R}^d$,  we introduce
$$\mathrm{diam}(\Lambda)=\sup_{n,n'\in \Lambda}|n-n'|, \ \mathrm{dist}(m,\Lambda)=\inf_{n\in \Lambda}|m-n|.$$


For $\Theta\in\R^d$ and $1\leq j\leq d$, let ${\Theta}_j^\neg=(\Theta_1,\cdots,\Theta_{j-1},\Theta_{j+1}\cdots,\Theta_d)\in\R^{d-1}$.

For $x\in\mathbb{R}^{d_1}$ and $\emptyset\neq X\subset\mathbb{R}^{d_1+d_2}$,  define the $x$-section of $X$ to be
$$X(x)=\{y\in\mathbb{R}^{d_2}:\  (x,y)\in X\}.$$ For example, $X({\Theta}_j^\neg)=\{\Theta_j\in\R:\ (\Theta_j,\Theta_j^\neg)\in X\}$ if $\emptyset \neq X\subset\R^d$.


Throughout this paper, we assume $\rho\in (0,1)$ for simplicity.

\subsection{Aubry duality}
It is easy to check that if
\begin{align*}
\Psi(x)=e^{\mathbf{i}\Theta\cdot x}\sum_{\mathbf{k}\in{\Z^{b}}}\Psi_\mathbf{k}e^{\mathbf{i} \mathbf{k}\cdot(\vec\theta+x\vec{\omega})}
\end{align*}
 is a solution of $(-\Delta+\varepsilon V(\vec\theta+x\vec\omega))\Psi(x)=E\Psi(x)$ (i.e. the Floquet-Bloch solution), then $\{\Psi_\mathbf{k}\}_{\mathbf{k}\in\Z^{b}}$ satisfies the following  equation
 \begin{align*}
 \sum_{\mathbf{k}'\in\Z^{b}}\varepsilon\widehat{V}_{\mathbf{k}-\mathbf{k}'}\Psi_{\mathbf{k}'}+\sum_{i=1}^d(\Theta_i+\mathbf{k}_i\cdot\vec\omega_i)^2\Psi_\mathbf{k}=E\Psi_\mathbf{k}.
 \end{align*}
This motivates us to study the following \textit{unbounded} quasi-periodic operators on $\Z^b$
\begin{align}\label{aubry}
(h(\Theta)Z)_\mathbf{k}= \sum_{\mathbf{k}'\in\Z^{b}}\varepsilon\widehat{V}_{\mathbf{k}-\mathbf{k}'}Z_{\mathbf{k}'}+\sum_{i=1}^d(\Theta_i+\mathbf{k}_i\cdot\vec\omega_i)^2Z_\mathbf{k}.
\end{align}
We call $h(\Theta)$ the Aubry duality  of $-\Delta+\varepsilon V(\vec\theta+x\vec\omega).$ For Aubry duality results in discrete case, we refer to \cite{BJ02,JK16,ShiG}.

\subsection{Green's functions and elementary regions}\label{gfer}
If $\Lambda\subset \Z^{b}$, denote ${h}_{\Lambda}(\Theta)=R_{\Lambda}{h(\Theta)}R_{\Lambda}$, where $R_{\Lambda}$ is the restriction operator and $h(\Theta)$ is given by \eqref{aubry}. Define the Green's function as
\begin{align*}
G_{\Lambda}(E;\Theta)=({h}_{\Lambda}(\Theta)-E+\mathbf{i}0)^{-1}.
\end{align*}

Denote by $Q_N$ an elementary region of size $N$ centered at $0$ (as in \cite{JLS20}):
\begin{align*}
  Q_N=[-N,N]^{b}\ {\rm or}\ Q_N=[-N,N]^{b}\setminus\{\mathbf{k}\in\mathbb{Z}^{b}: \ {k}_{ij}\in\varsigma_{ij}, \ 1\leq i\leq d, 1\leq j\leq b_i\},
\end{align*}
where  $ \varsigma_{ij}\in \{\{n<0\}, \{n>0\},\emptyset\}$ and at least two $ \varsigma_{ij}$  are not $\emptyset$.
Denote by $\mathcal{E}_N^{0}$ the set of all elementary regions of size $N$ centered at $0$. Let $\mathcal{E}_N$ be the set of all translates of  elementary regions:
$\mathcal{E}_N=\bigcup\limits_{\mathbf{k}\in\mathbb{Z}^{b},Q_N\in \mathcal{E}_N^{0}}\{\mathbf{k}+Q_N\}.$
\subsection{Semi-algebraic sets}
\begin{defn}[Chapter 9, \cite{BB}]
A set $\mathcal{S}\subset \mathbb{R}^n$ is called a semi-algebraic set if it is a finite union of sets defined by a finite number of polynomial equalities and inequalities. More precisely, let $\{P_1,\cdots,P_s\}\subset\mathbb{R}[x_1,\cdots,x_n]$ be a family of real polynomials whose degrees are bounded by $d$. A (closed) semi-algebraic set $\mathcal{S}$ is given by an expression
\begin{equation}\label{smd}
\mathcal{S}=\bigcup\limits_{j}\bigcap\limits_{\ell\in\mathcal{L}_j}\left\{x\in\mathbb{R}^n: \ P_{\ell}(x)\varsigma_{j\ell}0\right\},
\end{equation}
where $\mathcal{L}_j\subset\{1,\cdots,s\}$ and $\varsigma_{j\ell}\in\{\geq,\leq,=\}$. Then we say that $\mathcal{S}$ has degree at most $sd$. In fact, the degree of $\mathcal{S}$ which is denoted by $\deg(\mathcal{S})$, means the  smallest $sd$ over all representations as in $(\ref{smd})$.
\end{defn}

\begin{lem}[Tarski-Seidenberg Principle, \cite{BB}]\label{tsp}
Denote by $(x,y)\in\mathbb{R}^{d_1+d_2}$ the product variable. If $\mathcal{S}\subset\mathbb{R}^{d_1+d_2}$ is semi-algebraic of degree $B$, then its projections $\mathrm{Proj}_x\mathcal{S}\subset\mathbb{R}^{d_1}$ and
 $\mathrm{Proj}_y\mathcal{S}\subset\mathbb{R}^{d_2}$ are semi-algebraic of degree at most $B^{C}$, where $C=C(d_1,d_2)>0$.
\end{lem}


\begin{lem}[\cite{Bou07}]\label{proj}
 Let $\mathcal{S}\subset[0,1]^{d=d_1+d_2}$ be a semi-algebraic set of degree $\deg(\mathcal{S})=B$ and $\mathrm{mes}_d(\mathcal{S})\leq\eta$, where
$\log B\ll \log\frac{1}{\eta}.$

Denote by $(x_1,x_2)\in[0,1]^{d_1}\times[0,1]^{d_2}$ the product variable. Suppose
$ \eta^{\frac{1}{d}}\leq\epsilon.$
Then there is a decomposition of $\mathcal{S}$ as
\begin{align*}
\mathcal{S}=\mathcal{S}_1\cup\mathcal{S}_2
\end{align*}
with the following properties. The projection of $\mathcal{S}_1$ on $[0,1]^{d_1}$ has small measure
$$\mathrm{mes}_{d_1}(\mathrm{Proj}_{x_1}\mathcal{S}_1)\leq {B}^{C(d)}\epsilon,$$
and $\mathcal{S}_2$ has the transversality property
\begin{align*}
\mathrm{mes}_{d_2}(\mathcal{L}\cap \mathcal{S}_2)\leq B^{C(d)}\epsilon^{-1}\eta^{\frac{1}{d}},
\end{align*}
where $\mathcal{L}$ is any $d_2$-dimensional hyperplane in $\R^d$ s.t.,
$\max\limits_{1\leq j\leq d_1}|\mathrm{Proj}_\mathcal{L}(e_j)|<{\epsilon},$
where we denote by $e_1,\cdots,e_{d_1}$ the $x_1$-coordinate vectors.
\end{lem}

\section{LDT for Green's functions}

The main result of this section is the following LDT for Green's functions.
\begin{thm}[LDT]\label{ldt}
There exists $c_{\star}(b,d)>0$ such that the following holds: For any $0<c_1\leq c_{\star}$, there exists $N_0=N_0(c_1,\rho,b,d)>0$  such that if
\begin{align*}
\log\log\frac{1}{\varepsilon}\geq N_0,
\end{align*}
then for all $N\geq N_0$ we have
\begin{itemize}
  \item[(i)] There is some semi-algebraic set  $\Omega_N=\Omega_N(c_1,\varepsilon,\rho,b,d) \subset [0,2\pi]^{b}$ with  $\deg(\Omega_N)\leq N^{4d}$, and as $\varepsilon \to 0$,
  \begin{align*}
    \mathrm{mes}\left([0,2\pi]^{b}\backslash \bigcap\limits_{N\geq {N}_0} \Omega_N\right) \to 0.
  \end{align*}
  \item[(ii)]
  If  $\vec\omega\in\Omega_N$ and $E\in\mathbb{R}$ satisfying $|E|\leq|\log\varepsilon|^{\frac{1}{2c_1}}$, then there exists some  set $\mathbf{X}_N=\mathbf{X}_N(\vec\omega,E)\subset \mathbb{R}^d$  such that
 \begin{align*}\sup_{1\leq j\leq d,\Theta_j^\neg\in \R^{d-1}}\mathrm{mes}(\mathbf{X}_N(\Theta_j^\neg))\leq e^{-N^{c_1}},\end{align*}
and for  $\Theta\notin \mathbf{X}_N$, $Q\in\mathcal{E}_N^{0}$,
\begin{align*}
 \|G_{Q}(E;\Theta)\|&\leq e^{\sqrt{N}},\\
|G_{Q}(E;\Theta)(\mathbf{n},\mathbf{n}')|&\leq  e^{-\frac{\rho}{10}|\mathbf{n}-\mathbf{n}'|}\  {\mathrm{for} \ |\mathbf{n}-\mathbf{n}'|\geq {N}/{10}}.
\end{align*}
\end{itemize}
\end{thm}


\begin{proof}[\bf Proof of Theorem \ref{ldt}]

The proof is based on the multi-scale analysis scheme as in \cite{Bou07,JLS20}. However, we formulate it into three steps.

{\bf STEP 1: The First Step}

In this initial step we use a perturbative argument. It is important that in this step we obtain\textit{ uniform} measure estimates both on $E\in\R$ and $\vec\omega.$
\begin{lem}\label{lemini}
Given $\delta>0$, let 
\begin{align*}
\mathbf{X}_{N}=\bigcup_{|\mathbf{k}|\leq {N}}\left\{\Theta\in\R^d:\  \left|\sum_{i=1}^d(\Theta_i+\mathbf{k}_i\cdot\vec\omega_i)^2-E\right|<\delta\right\}.
\end{align*}
Then
\begin{align*}
\sup_{1\leq j\leq d,\ \Theta_j^\neg\in\R^{d-1}}\mathrm{mes}(\mathbf{X}_N(\Theta_j^\neg))\leq C(2N+1)^{b}\sqrt{\delta},
\end{align*}
where $C>0$ is an absolute constant.

Moreover, if
$$\varepsilon^{-1}\geq 2\delta^{-1}(2N+1)^{b},$$
then for any $\Theta\notin \mathbf{X}_{N}$ and $\Lambda\subset [-N,N]^{b}$,
\begin{align*}
\|G_{\Lambda}(E;\Theta)\|&\leq 2\delta^{-1}, \\
|G_{\Lambda}(E;\Theta)(\mathbf{n},\mathbf{n}')|&\leq 2\delta^{-1}e^{-{\rho}|\mathbf{n}-\mathbf{n}'|}.
\end{align*}
\end{lem}
\begin{proof}
The proof is perturbative and based on the Neumann series argument: The proof of measure bound is based on  Lemma 4.7 in \cite{Shi19}, and the estimate on Green's functions can be found in the proof of Theorem 4.3 in
\cite{JLS20}.
\end{proof}

If we take $C(2N+1)^{b}\sqrt{\delta}=e^{-N^{c_1}}$ in the above lemma, we obtain
\begin{prop}\label{inip}
Let $0<c_1<1/4$. Then there exists  $N_0=N_0(c_1,\rho,b,d)>0$ such that the following: If $N_0\leq N\leq |\log\varepsilon|^{\frac{1}{2c_1}}$, then for all $E\in\R$ and $\vec\omega\in[0,2\pi]^{b}$ there is some $\mathbf{X}_N=\mathbf{X}_N(E,\vec\omega)\subset\R^d$ such that
\begin{align}\label{imb}
\sup_{1\leq j\leq d, \ \Theta_j^\neg\in\R^{d-1}}\mathrm{mes}(\mathbf{X}_N(\Theta_j^\neg))\leq e^{-N^{c_1}},
\end{align}
and, if $\Theta\notin \mathbf{X}_{N}$ and $Q\in \mathcal{E}_N^0$,
\begin{align}
\label{numan1}
\|G_{Q}(E;\Theta)\|&\leq e^{\sqrt{N}}, \\
\label{numan}
|G_{Q}(E;\Theta)(\mathbf{n},\mathbf{n}')|&\leq e^{-\frac{4\rho}{5}|\mathbf{n}-\mathbf{n}'|}\ {\rm for}\ |\mathbf{n}-\mathbf{n}'|\geq N/10.
\end{align}
\end{prop}

\begin{proof}
Let $\mathbf{X}_N$ be given by Lemma \ref{lemini}. Then the measure bound \eqref{imb} can be derived from choosing $C(2N+1)^{b}\sqrt{\delta}=e^{-N^{c_1}}$, i.e.,
\begin{align}\label{delta}
\delta=C^{-2}(2N+1)^{-2b}e^{-2N^{c_1}}.
\end{align}

We then turn to the Green's function estimates. Let $\Theta\not\in\mathbf{X}_N.$
Since \eqref{delta},  the condition $\varepsilon^{-1}\geq \delta^{-1}(2N+1)^{b}$ is equivalent to
\begin{align*}
\varepsilon^{-1} \geq 2C^{2}(2N+1)^{3b}e^{2N^{c_1}}
\end{align*}
It suffices to let $N\geq N_0(c_1,b,d)>0$ and
\begin{align}\label{epsilon}
\varepsilon^{-1}\geq e^{{N}^{2c_1}}\geq2C^{2}(2N+1)^{3b}e^{2N^{c_1}},
\end{align}
that is $N_0(c_1,b,d)\leq N\leq |\log\varepsilon|^{\frac{1}{2c_1}}$.
From  \eqref{epsilon} and $0<c_1<1/4$, we have
$2\delta^{-1}\leq  e^{{N}^{2c_1}}\leq  e^{\sqrt{N}}$, which implies \eqref{numan1}.
 Finally, the exponential decay bound \eqref{numan} follows from: If $|\mathbf{n}-\mathbf{n}'|\geq N/10$ and $N\geq N_0(c_1,\rho,b,d)>0,$ then
\begin{align*}
|G_{Q}(E;\Theta)(\mathbf{n},\mathbf{n}')|\leq e^{\sqrt{N}-\frac{\rho}{50}N}e^{-\frac{4\rho}{5}|\mathbf{n}-\mathbf{n}'|}\leq e^{-\frac{4\rho}{5}|\mathbf{n}-\mathbf{n}'|}.
\end{align*}
\end{proof}

{\bf STEP 2: The Second Step}

In the second step we will propagate LDT from (scales) $\left[N_0,|\log\varepsilon|^{\frac{1}{2c_1}}\right]$ to
$$\left[|\log\varepsilon|^{\frac{1}{2c_1}}, e^{|\log\varepsilon|^{1/4}}\right].$$
For this purpose, further removal of $\vec\omega$ is necessary  and the semi-algebraic geometry arguments established by Bourgain \cite{Bou07} play a key role.  In what follows we always assume
\begin{align}\label{ebd}
E\in \mathcal{I}_\varepsilon= \left[-|\log \varepsilon|^{\frac{1}{2c_1}}, \log \varepsilon|^{\frac{1}{2c_1}}\right].
\end{align}

Denote $\mathbf{B}(N)=\{\Theta\in\R^d:\ |\Theta|\leq N\}$. Then we have
\begin{prop}\label{inism}
Let $|\log\varepsilon|^{\frac{1}{2c_1}}\leq N\leq e^{|\log\varepsilon|^{1/4}}$,  $N_1\sim (\log N)^{2/{c_1}}$. Then there exist constants
$$0<c_2(b,d)<c_3(b,d)<c_4(b,d)\ll1,$$
and semi-algebraic set $\Omega_N\subset[0,2\pi]^{b}$ satisfying
$$\deg(\Omega_N)\leq N^{4d},\ {\rm mes}([0,2\pi]^{b}\setminus\Omega_N)\leq N^{-c_2}$$
 such that the following holds: If $\log\log\frac{1}{\varepsilon}\geq N_0(c_1,b,d)>0$ and $\vec\omega\in\Omega_N$, then for all $(E,\Theta)\in \mathcal{I}_\varepsilon\times\mathbf{B}(100bN^2)$  there is some $\frac{1}{10}N^{c_3}<M<10N^{c_4}$
such that for all $\mathbf{k}\in [-M,M]^{b}\setminus[-M^{\frac{1}{10b}}, M^{\frac{1}{10b}}]^{b}$, one has $\Theta+\mathbf{k}\vec\omega\notin \mathbf{X}_{N_1},$ here $\mathbf{X}_{N_1}$ is given by Proposition \ref{inip}.
\end{prop}

\begin{rem}
The conditions $\log\log \frac{1}{\varepsilon}\geq N_0(c_1,b,d)>0$ and $N\leq e^{|\log\varepsilon|^{1/4}}$ can ensure that $N_0(c_1,\rho,b,d)\leq N_1\leq |\log\varepsilon|^{\frac{1}{2c_1}}$.
\end{rem}
\begin{proof}

We will eliminate the variables $(\Theta,E)\in \mathbf{B}(100bN^2)\times \mathcal{I}_\varepsilon$. This needs make \textit{quantitative} restrictions on $\vec\omega$.

 Let
\begin{align*}
{N}^{c_3}<L< {N}^{c_4},
\end{align*}
where $0<c_3<c_4<1$ will be specified later. Let $S_L$ be the set of all $(\vec\omega,\Theta, E)\in [0,2\pi]^{b}\times \mathbf{B}(100bN^2)\times \mathcal{I}_\varepsilon$ such that, for any $\mathbf{n}\in [-L,L]^{b}$ and $\vec\omega\in [0,2\pi]^{b},$  $\Theta+\mathbf{n} \vec\omega\in \mathbf{X}_{N_2}(E,\vec\omega).$ Applying Proposition \ref{inip}  gives
\begin{align}
\label{ldt5}&\deg(S_L)\leq L^{C(b,d)},\\
\label{ldt6}&\sup_{1\leq j\leq d,\Theta_j^\neg\in\mathbb{R}^{d-1}}\mathrm{mes}(\mathbb{R}\setminus  S_L(\Theta_{j}^\neg))\leq e^{-{N_1^{c_1}}/{2}}.
\end{align}
 Fix $I\subset\{1,\cdots,d\}$ and define $\mathcal{A}$ to be all $(\vec\omega,\Theta,y,E)\in {[0,2\pi]}^{b}\times \mathbf{B}(100bN^2)\times \mathbb{R}^{I}\times \mathcal{I}_\varepsilon$ satisfying
\begin{align}\label{ldt9}
(\vec\omega,(\Theta_j+y_j)_{j\in {I}},(\Theta_j)_{j\notin {I}}, E)\notin S_L.
\end{align}
Obviously, by (\ref{ldt5}) and (\ref{ldt9}),
\begin{align}\label{ldt10}\deg(\mathcal{A})\leq L^{C(b,d)}.
\end{align}
Fix $\vec\omega\in [0,2\pi]^{b}$ and consider
$$\mathcal{A}_1:=\mathcal{A}(\vec\omega)\subset\mathbf{B}(100bN^2)\times \mathbb{R}^{I}\times \mathcal{I}_\varepsilon.$$
Assume for $y_I=(y_i)_{i\in I}$, $y_I\notin\mathbf{B}_I(200bN^2)=\{y\in\R^{I}:\ |y|\leq 200bN^2\}$.  Then for all $\mathbf{n}\in[-L,L]^{b}+\mathcal{E}_{N_1}^0$, one has since $|\Theta|\leq 100bN^2$ and $|E|\leq |\log \varepsilon|^{\frac{1}{2c_1}}\leq N$,
\begin{align}\label{nres}
\nonumber\sum_{i=1}^d(\Theta_i+y_i+\mathbf{n}_i\cdot\vec\omega_i)^2-E&\geq\sum_{i\in {I}}(\Theta_i+y_i+\mathbf{n}_i\cdot\vec\omega_i)^2-E\\
&\geq N^4,
\end{align}
which shows that there is \textit{no resonances} in this case. In other words, we must have
\begin{align}
\mathcal{A}_1\subset\mathbf{B}(100bN^2)\times \mathbf{B}_I(200bN^2)\times \mathcal{I}_\varepsilon.
\end{align}

By (\ref{ldt10}) and Tarski-Seidenberg principle (i.e., Lemma \ref{tsp}), we obtain
\begin{align}\label{a1d}
\deg(\mathcal{A}_1)\leq L^{C(b,d)}.
\end{align}
From (\ref{ldt6}), for all $(\Theta,E)\in\mathbf{B}(100bN^2)\times\mathcal{I}_\varepsilon$, we have
\begin{align}\label{ldt11}
\mathrm{mes}_{I}(\mathcal{A}_{1}(\Theta,E))\leq \eta:=e^{-{N_1^{c_1}}/{2}}.
\end{align}

At this stage, we need a lemma for eliminating  multi-variable:
\begin{lem}\label{svl}
Let $\mathcal{S}\subset [0,1000bN^2]^{s+r}$ be a semi-algebraic set of degree $B$ and such that
\begin{align*}
\mathrm{mes}_s(\mathcal{S}(y))<\eta\ \mathrm{for}\ \forall\  y\in [0,1000bN^2]^r\ {\rm and }\ \log(BN)\ll\log{\frac{1}{\eta}}.
\end{align*}
Then the set
\begin{align*}
\left\{(x_1,\cdots,x_{2^r})\in [0,1000bN^2]^{s2^r}:\  \bigcap\limits_{1\leq i\leq 2^r}\mathcal{S}(x_i)\neq\emptyset\right\}
\end{align*}
is semi-algebraic of degree at most $B^{C}$ and measure at most
\begin{align*}
N^CB^{C}\eta^{s^{-r}2^{-r(r-1)/2}},
\end{align*}
where $C=C(s,r,b)>0$.
\end{lem}
\begin{proof}[\bf Proof of Lemma \ref{svl}]
The degree bound follows from Tarski-Seidenberg principle Lemma \ref{tsp}. The measure bound is derived as follows. We first divide $[0,1000bN^2]^{s+r}$ into about $N^{C(s,r,b)}$ many unit cubes. Then applying Lemma 1.18 of \cite{Bou07} gives measure bound $B^{C}\eta^{s^{-r}2^{-r(r-1)/2}}$ on each such unit cube. Finally, it suffices to take account of all those cubes.
\end{proof}
From \eqref{a1d} and (\ref{ldt11}),   we have by using   Lemma \ref{svl} (with $s=|I|$, $r=d+1, B=L^C$ and $\eta=e^{-N_1^{c_1}/2}$) that

$\mathcal{A}_{2}=\left\{(y^{(i)})_{1\leq i\leq 2^{d+1}}:\  \bigcap_{1\leq i\leq 2^{d+1}}\mathcal{A}_{1}(y^{(i)})\neq \emptyset\right\}$
 is a semi-algebraic set with
\begin{align}\label{ldt12}\deg(\mathcal{A}_{2})\leq L^{C(b,d, |I|)},\ \mathrm{mes}(\mathcal{A}_{2})\leq\eta_1:= {N}^{C(b,d,|I|)}\eta^{|{I}|^{-{d-1}}2^{-d(d+1)/2}}.\end{align}
Define
\begin{align*}\mathcal{B}&=\left\{\left(\vec\omega,(y^{(i)})_{1\leq i\leq 2^{d+1}}\right):\  \vec\omega\in [0,2\pi]^{b}, \bigcap_{1\leq i\leq 2^{d+1}}\mathcal{A}(\vec\omega,y^{(i)})\neq\emptyset \right\}\\
&\subset [0,2\pi]^{b}\times[0,200b{N}^2]^{|{I}|2^{d+1}}.
\end{align*}
By Lemma \ref{tsp}, $\deg(\mathcal{B})\leq L^{C(d,|I|)}$. Write ${\overline{{\omega}}}=(\vec\omega_i)_{i\notin{I}}$ and $\widetilde{\omega}:=(\vec\omega_i)_{i\in {I}}$.
Then by Fubini Theorem and  (\ref{ldt12}),
\begin{eqnarray}\label{ldtn}
\mathrm{mes}(\mathcal{B}(\overline{\omega}))\leq\eta_1,\ \deg(\mathcal{B}({\overline{\omega}}))\leq L^{C}.
\end{eqnarray}
Notice that for any $C_\star>1$,   one has  $\frac{1}{\eta_1}\gg{N}^{C}$ if $N\geq N_0(C_\star,c_1,b,d,|I|)>0$.
One considers the set $\mathcal{B}_1$ containing $\widetilde{\omega}$, which is defined by the following:
\textit{there is  some sequence} $\mathbf{n}^{(1)},\cdots,\mathbf{n}^{(2^{d+1})}\in \mathbb{Z}^{b_I} (b_I=\sum_{i\in I}b_i)$ \textit{satisfying}
{
\begin{eqnarray}
\label{stp2}&&L^{{\underline{C}}_{\alpha}}\leq\min_{i\in {I},1\leq j\leq b_i}|{n}^{(\alpha)}_{ij}|\leq |\mathbf{n}^{(\alpha)}|\leq L^{\overline{C}_{\alpha}} \ (1\leq \alpha\leq 2^{d+1}),
\end{eqnarray}
\textit{such that}
\begin{equation*}
(\widetilde{\omega},\mathbf{n}^{(1)}\widetilde{\omega} ,\cdots,\mathbf{n}^{(2^{d+1})}\widetilde{\omega})\in\mathcal{B}{(\overline{\omega})},
\end{equation*}
\textit{where}
$$1\ll \underline{C}_1\ll\overline{ C}_1\ll\cdots\ll \underline{C}_{2^{d+1}}\ll \overline{C}_{2^{d+1}}.$$
}

At this stage, we will introduce a key lemma for eliminating frequencies:
\begin{lem}\label{rmf}
Let $\mathcal{S}\subset \R^{sr}$ be a semi-algebraic set of degree $B$ and $\mathrm{mes}(\mathcal{S})<\eta$ with $\eta>0.$

Let $b(r)=\sum_{i=1}^rb_i,b_i\in\N.$ For $\vec\omega_i\in[0,2\pi]^{b_i}$,  $\vec\omega= (\vec\omega_1,\cdots,\vec\omega_r)\in[0,2\pi]^{b(r)}$ and $ \mathbf{n}=(\mathbf{n}_1,\cdots,\mathbf{n}_r)\in\mathbb{Z}^{b(r)} (\mathbf{n}_i\in\Z^{b_i})$, define
$$\mathbf{n}\vec\omega=(\mathbf{n}_1\cdot\vec\omega_1,\cdots, \mathbf{n}_r\cdot\vec\omega_r).$$

For any $C>1$, define  $\mathcal{N}_1,\cdots,\mathcal{N}_{s-1}\subset \mathbb{Z}^{b(r)}$ to be finite sets with the following property:
$$\min\limits_{1\leq i\leq r, 1\leq j\leq b_i}|{n}_{ij}|>(B\max\limits_{1\leq i\leq r}|\mathbf{m}_i|)^{C}, $$
where $\mathbf{n}\in \mathcal{N}_{l}, \mathbf{m}\in\mathcal{N}_{l-1}\  (2\leq l\leq s-1)$.

 Then there is some $C=C(r,s,b)>0$ such that for
$\max\limits_{\mathbf{n}\in\mathcal{N}_{s-1}}|\mathbf{n}|^C<\frac{1}{\eta},$
one has
\begin{align*}
&\ \mathrm{mes}(\{\vec\omega\in [0,2\pi]^{b(r)}:\  \exists \ \mathbf{n}^{(i)}\ \in\mathcal{N}_i\  s.t.,\  (\vec\omega,\mathbf{n}^{(1)}\vec\omega,\cdots,\mathbf{n}^{(s-1)}\vec\omega)\in \mathcal{S}\})\\
&\ \ \ \ \ \leq B^{C}(\min\limits_{\mathbf{n}\in\mathcal{N}_1}\min\limits_{ i,j}|{n}_{ij}|)^{-1}.
\end{align*}
\end{lem}

\begin{proof}[\bf Proof of Lemma \ref{rmf}]
The proof is similar to that of Lemma 1.20 in \cite{Bou07} by iterating  Lemma \ref{proj}. The main difference is that we allow $\mathcal{S}$ to vary in $\R^{sr}$ rather than in the unit cube. Bourgain's proof remains applicable since Lemma \ref{proj} permits re-scaling. Moreover,  at $i$-th ($i\leq s-1$) iteration step, the valid sets are essentially contained  in $[0, \max\limits_{\mathbf{n}\in\mathcal{N}_{s-i}}|\mathbf{n}|]^{r(s-i+1)}$. We omit the details here.

\end{proof}

Then  from (\ref{ldtn}) and  Lemma \ref{rmf} (set $r=|I|, s=2^{d+1}+1$), we have
\begin{eqnarray}\label{ldt14}
&& \mathrm{mes}(\mathcal{B}_1)\leq L^{-{\underline{C}_1}}L^{C}\leq L^{-3}.
\end{eqnarray}
Define  $$\mathcal{G}_L:=\bigcap\limits_{\emptyset\neq {I}\subset\{1,\cdots,d\}
}\left\{\vec\omega\in [0,2\pi]^{b}:\  \vec\omega=(\overline{\omega},\widetilde{\omega}),\  \widetilde{\omega}\notin\mathcal{B}_1\right\}.$$
Thus \textit{$\vec\omega\notin\mathcal{G}_{L}$ if and only if, there are  $\emptyset\neq{I}\subset \{1,\cdots,d\}$  and  some sequence $\mathbf{n}^{(1)},\cdots,\mathbf{n}^{(2^{d+1})}\in \mathbb{Z}^{b_I}$ satisfying \eqref{stp2} such that,
 $$(\vec\omega,\mathbf{n}^{(1)}\widetilde{\omega},\cdots, \mathbf{n}^{(2^{d+1})}\widetilde{\omega})\in \mathcal{B}.$$}
 This implies that $\mathcal{G}_{L}$ is a semi-algebraic set. Furthermore, by (\ref{stp2}) and (\ref{ldt14}),
 \begin{eqnarray*}\label{gl1}
\mathrm{mes}([0,2\pi]^{b}\setminus\mathcal{G}_{L})\leq L^{-2},\ \deg(\mathcal{G}_{L})\leq L^{C\overline{C}_{2^{d+1}}}.
\end{eqnarray*}
We should remark that $\mathcal{G}_L$ also depends on $\underline{C}_1,\cdots,\overline{C}_{2^{d+1}}$.

Finally, for $0<c_3(b,d)\ll c_4(b,d)\ll1$, we can choose appropriate $L_l, \underline{C}_1,\cdots,\overline{C}_{2^{d+1}}$ and then iterate along every axis direction (i.e., induction on $I$) as done by Bourgain \cite{Bou07} (see the proof of the \textsc{Claim})
if  $\vec\omega\in\mathcal{G}_{L_l}$. The number of all possible  $L_l, \underline{C}_1,\cdots,\overline{C}_{2^{d+1}}$ is finite and depends only on $b,d$. In particular, we have
\begin{align*}
\deg(\Omega_N)\leq CN^{c_4C\overline{C}_{2^{d+1}}}&\leq N^{4d},\\
 {\rm mes}([0,2\pi]^{b}\setminus\Omega_N)\leq CN^{-2c_3}&\leq N^{-c_2},
\end{align*}
where $\Omega_N=\bigcap\limits_{L_l,\underline{C}_1,\cdots,\overline{C}_{2^{d+1}}}\mathcal{G}_{L_l}$ and $ 0<c_2(b,d)\ll c_3(b,d)$. In addition, if $\vec\omega\in \Omega_N$, then for all $(E,\Theta)\in  \mathcal{I}_\varepsilon\times\mathbf{B}(100bN^2)$  there is some $\frac{1}{10}N^{c_3}<M<10N^{c_4}$
such that for all $\mathbf{k}\in [-M,M]^{b}\setminus[-M^{\frac{1}{10b}}, M^{\frac{1}{10b}}]^{b}$, one has $\Theta+\mathbf{k}\vec\omega\notin \mathbf{X}_{N_1}.$


\end{proof}

Combining the above proposition and Cartan's estimate on subharmonic functions, we can finish the proof of LDT in $[N, N^2]$.

Since \eqref{ebd} and \eqref{nres}, it suffices to consider in case
\begin{align}\label{Tbd}
\Theta\in \mathbf{B}(1000bN^2).
\end{align}

\begin{prop}\label{xn}
  Assume the assumptions of Proposition \ref{inism} are satisfied.  Let $\Theta$ satisfy \eqref{Tbd}.  Then  for  $\vec\omega\in\Omega_N$, we have
  \begin{itemize}
  \item[(1)] Fix $1\leq j\leq d$ and $\Theta_j^\neg\in \mathbf{B}_{d-1}(1000bN^2)$. Write $\Theta=(\Theta_j, \Theta_j^\neg)\in\R^{d}$. Assume there exist $\widetilde N\in[{N^{c_3}}/{4}, N^{c_4}]$ and $\bar{\Lambda}\subset \Lambda\in \mathcal{E}_{\widetilde N}$ with ${\rm diam}(\bar \Lambda)\leq 4\widetilde N^{\frac{1}{10b}}$ and $\Lambda\subset \mathbf{B}(1000bN^2)$ such that,
for any $ \mathbf{k}\in  \Lambda \backslash  \bar{\Lambda}$,
there exists some $\mathcal{E}_{N_1}\ni W\subset \Lambda\backslash  \bar{\Lambda}$ such that ${\rm dist}(\mathbf{k},\Lambda \backslash  \bar{\Lambda}\backslash W)\geq {N_1}/{2},$ and $\Theta+\mathbf{k}\vec\omega\notin \mathbf{X}_{N_1}$.
Let
\begin{align*}
  Y_\Theta=\left\{y\in \R: |y-\Theta_j|\leq e^{-10\rho N_1},\ |\Theta_j|\leq 1000b{N}^2,\ \|G_{\Lambda}(E;(y,\Theta_j^\neg))\|\geq e^{\sqrt{{\widetilde N}}}\right\}.
\end{align*}
Then
\begin{align*}
\mathrm{mes}(Y_\Theta)\leq e^{-{\widetilde N}^{\frac{1}{3}}}.
\end{align*}

\item[(2)] Fix $N_\star\in[N,N^2]$. If $E\in \mathcal{I}_\varepsilon$ and $0<c_1<c_3/10$, then there exists some  set $\mathbf{X}_{N_\star}=\mathbf{X}_{N_\star}(E,\vec\omega)\subset \mathbb{R}^d$  such that
 \begin{align*}\sup_{1\leq j\leq d,\Theta_j^\neg\in \R^{d-1}}\mathrm{mes}( \mathbf{X}_{N_\star}(\Theta_j^\neg))\leq e^{-N_\star^{c_1}},\end{align*}
and for $\Theta\notin \mathbf{X}_{N_\star}$, $Q\in\mathcal{E}_{N_\star}^{0}$,
\begin{align*}
 |G_{Q}(E;\Theta)(\mathbf{n},\mathbf{n}')|&\leq  e^{-(\frac{4}{5}\rho-\frac{C(\rho,b,d)}{\sqrt{N_1}})|\mathbf{n}-\mathbf{n}'|}\  {\mathrm{for} \ |\mathbf{n}-\mathbf{n}'|\geq {N_\star}/{10}}.
\end{align*}
\end{itemize}
\end{prop}

\begin{proof}

(1)  Let $\mathcal{D}$ be the $e^{-10\rho N_1}$ neighbourhood  of $\Theta_j$ in the complex plane, i.e.,
\begin{align*}
 \mathcal{D}=\{y\in \mathbb{C}:\  |\Im y|\leq e^{-10\rho N_1}, |\Re y-\Theta_j|\leq e^{-10\rho N_1}\}.
\end{align*}
From assumptions, we have for all $\mathbf{k}\in \Lambda\backslash \bar{\Lambda}$ and $Q\in \mathcal{E}_{N_1}^0$,
\begin{align}
\|G_{Q}(E;\Theta+\mathbf{k}\vec\omega)\|&\leq e^{\sqrt{N_1}},\label{Apr20}\\
|G_{Q}(E;\Theta+\mathbf{k}\vec\omega)(\mathbf{n},\mathbf{n}')|&\leq e^{-\frac{4\rho}{5}|\mathbf{n}-\mathbf{n}'|}\ {\rm for}\ |\mathbf{n}-\mathbf{n}'|\geq{N_1}/{10}.\label{Apr21}
\end{align}
Note that for all $\mathbf{n},\mathbf{n}'\in [-N_1,N_1]^{b}$,
\begin{align*}
e^{-10\rho N_1}<e^{-3\rho N_1-\rho|\mathbf{n}-\mathbf{n}'|}.
\end{align*}
Then by Lemma \ref{pag},  \eqref{Apr20} and \eqref{Apr21}, we have for any $y\in \mathcal{D}$, $Q\in \mathcal{E}_{N_1}^0$
and $\mathbf{k}\in\Lambda\backslash\bar\Lambda$,
\begin{align}
\|G_{Q}(E;(\Theta_j+y,\Theta_j^\neg)+\mathbf{k}\vec\omega)\|&\leq 2e^{\sqrt{N_1}},\label{Apr22}\\
|G_{Q}(E;(\Theta_j+y,\Theta_j^\neg)+\mathbf{k}\vec\omega)(\mathbf{n},\mathbf{n}')|&\leq 2e^{-\frac{4\rho}{5}|\mathbf{n}-\mathbf{n}'|}\ {\rm for}\ |\mathbf{n}-\mathbf{n}'|\geq{N_1}/10.\label{Apr23}
\end{align}
 Applying  Lemma \ref{res1} with $M_1=M_0=N_1$ implies for any $y\in \mathcal{D}$,
\begin{align}
\label{bs3}\|G_{\Lambda\setminus \bar\Lambda}(E;(\Theta_j+y,\Theta_j^\neg))\|\leq 4(2N_1+1)^{b}e^{\sqrt{N_1}}\leq e^{2\sqrt{N_1}}.
\end{align}

We then can use the following matrix-valued Cartan's estimate to propagate the ``smallness of measure'':

\begin{lem}[Cartan's estimate, \cite{BB}]\label{mcl}
Let $T(\theta)$ be a self-adjoint $N\times N$ matrix function of a parameter $\theta\in[-\delta,\delta]$  satisfying the following conditions:
\begin{itemize}
\item[(i)] $T(\theta)$ is real analytic in $\theta\in [-\delta,\delta]$ and has a holomorphic extension to
\begin{align*}
\mathcal{D}_{\delta,\delta_1}=\left\{\theta\in\mathbb{C}: \ |\Re \theta|\leq\delta,|\Im{\theta}|\leq \delta_1\right\}
\end{align*}
satisfying $\sup_{\theta\in \mathcal{D}_{\delta,\delta}}\|T(\theta)\|\leq K_1, K_1\geq 1.$

\item[(ii)]  For all $\theta\in[-\delta,\delta]$, there is subset $V\subset [1,N]$ with $|V|\leq M$ such that
\begin{align*}
\|(R_{[1,N]\setminus V}T(\theta)R_{[1,N]\setminus V})^{-1}\|\leq K_2, K_2\geq 1.
\end{align*}
\item[(iii)] Assume
\begin{align*}
\mathrm{mes}\{\theta\in[-{\delta}, {\delta}]: \ \|T^{-1}(\theta)\|\geq K_3\}\leq 10^{-3}\delta(1+K_1)^{-1}(1+K_2)^{-1}.
\end{align*}
\end{itemize}
Let $0<\epsilon\leq (1+K_1+K_2)^{-10 M}.$ Then
\begin{align}\label{mc5}
\mathrm{mes}\left\{\theta\in\left[-{\delta}/{2}, {\delta}/{2}\right]:\  \|T^{-1}(\theta)\|\geq \epsilon^{-1}\right\}\leq C\delta e^{-\frac{c\log \epsilon^{-1}}{M\log(K_1+K_2+K_3)}},
\end{align}
where $C, c>0$ are some absolute constants.
\end{lem}
We will apply  Lemma \ref{mcl}
with  \begin{align}\label{sep17}
T(y)=h_{\Lambda}((\Theta_j+y,\Theta_j^\neg))-{E},\delta=\delta_1=2e^{-10\rho N_1}.
\end{align}
It suffices to verify the assumptions of Lemma \ref{mcl}.
By assumptions \eqref{Tbd}, $\Lambda\subset \mathbf{B}(1000bN^2)$ and \eqref{bs3}, one has
\begin{align}\label{mb}
K_1=O(N^2),  M=|\bar\Lambda|\leq C(b,d)\widetilde N^{1/{10}},
K_2=e^{2\sqrt{N_1}}.
\end{align}
Since LDT holds at scale $N_2$ for $y$ being outside a set  of measure at most $e^{-{N_2^{c_1}}}$.
Applying Lemma \ref{res1} yields
\begin{eqnarray*}
 \|T^{-1}(y)\|\leq 4(2N_2+1)^{b}e^{\sqrt{N_2}}\leq e^{2\sqrt{N_2}}=K_3
\end{eqnarray*}
for $y$ being outside a set  of measure at most
$$(2\widetilde N+1)^{b}e^{-{N_2^{c_1}}}\leq e^{-{N_2^{c_1}}/{2}}.$$
It follows from $100N_1<N_1^{3/2}<N_2^{c_1}$ that
\begin{equation*}
10^{-3}\delta_1(1+K_1)^{-1}(1+K_2)^{-1}\geq e^{-{N_2^{c_1}}/{2}}.
\end{equation*}
This verifies (iii) of Lemma \ref{mcl}.
For $\epsilon=e^{-\sqrt{{\widetilde N}}}$,  one has by \eqref{sep17} and \eqref{mb},
  $\epsilon<(1+K_1+K_2)^{-10M}.$
By \eqref{mc5} of Lemma \ref{mcl},
\begin{equation*}
\mathrm{mes}(Y_\Theta)\leq e^{-\frac{c\sqrt{{\widetilde N}}}{N_2\widetilde N^{1/{10}}\log \widetilde N}}\leq e^{-\widetilde N^{1/3}}.
\end{equation*}

(2) Fix $N_\star\in[N,N^2]$ and $E\in \mathcal{I}_\varepsilon$. As done in \eqref{nres}, to prove LDT at scale $N_\star$ it suffices to restrict $\Theta\in \mathbf{B}(10bN_\star)\subset \mathbf{B}(10bN^2)$.

Fix $1\leq j\leq d,\Theta_j^\neg\in\R^{d-1}$ and  $\Theta=(\Theta_j,\Theta_j^\neg)\in\mathbf{B}(10bN_\star)$.  Then $\Theta+\mathbf{n}\vec\omega\in \mathbf{B}(100bN^2)$ for all $|\mathbf{n}|\leq N_\star$.
By using Proposition \ref{inism} ( with $\Theta$ replaced by $\Theta+\mathbf{n}\vec\omega$), for such $\Theta$ and any $\mathbf{n}\in Q\in\mathcal{E}_{N_\star}^0$, there exist $\frac{1}{4}N^{c_3}\leq \widetilde{N} \leq N^{c_4}$,  $\Lambda \in \mathcal{E}_{\widetilde{N}}$ and $\bar{\Lambda}$, such that
\begin{align*}
\mathbf{n}\in \bar\Lambda\subset\Lambda \subset Q, {\rm dist}(\mathbf{n},  Q\backslash \Lambda)\geq{\widetilde{N}}/{2},\ {\rm diam}(\bar{\Lambda}) \leq 4 \widetilde{N}^{\frac{1}{10b}}.
\end{align*}
Moreover,  for any $ \mathbf{k}\in  \Lambda\backslash  \bar{\Lambda}$, $\Theta+\mathbf{k}\vec\omega\notin \mathbf{X}_{N_1}$ and
there exists some $\mathcal{E}_{N_1} \ni W \subset \Lambda\backslash  \bar{\Lambda}$ such that $\mathbf{k}\in W,\  {\rm dist}(\mathbf{k},\Lambda\backslash  \bar{\Lambda}\backslash W)\geq {N_1}/{2}.$
We should remark that in the above argument we had applied essentially methods of \cite{JLS20} in dealing with elementary regions (of size $\widetilde N$) near the boundary of $Q$.  

We now fix above $\widetilde N, \bar\Lambda,\Lambda$ throughout the set $\{(y,\Theta_j^\neg)\in\R^d:\ |y-\Theta_j|\leq e^{-10\rho N_1}\}$. Recalling Lemma \ref{pag} and the above constructions, we have by (1) of Proposition \ref{xn} that 
 there exists a  $Y\subset \{y\in\R:\ |y-\Theta_j|\leq e^{-10\rho N_1}\}$ such that
\begin{equation}\label{Apr6}
\mathrm{mes}(Y)\leq e^{-\widetilde{N}^{{1}/{3}}},
\end{equation}
and for $\Theta_j\notin Y$, $\|G_{\Lambda}(E;\Theta)\|\leq e^{\sqrt{\widetilde{N}}}.$
Applying  Lemma \ref{res2} yields 
\begin{equation*}\label{Apr8}
  |G_{\Lambda}(E;\Theta)(\mathbf{n},\mathbf{n}')|\leq  e^{- (\frac{4\rho}{5}-\frac{C}{\sqrt{N_1}})|\mathbf{n}-\mathbf{n}'|}\  {\mathrm{for} \ |\mathbf{n}-\mathbf{n}'|\geq {\widetilde{N}}/{10}}.
\end{equation*}

Cover $[0,10bN_\star]$ by pairwise disjoint $e^{-10\rho N_1}$-size intervals and
let
\begin{equation}\label{Apr9}
 \mathbf{X}_{N_\star}(\Theta_j^\neg)=\bigcup_{Q\in\mathcal{E}_{N_\star}^0, \mathbf{n}\in Q, \Theta=(\Theta_j,\Theta_j^\neg)}
 Y.
\end{equation}
We remark that while $\Theta=(\Theta_j,\Theta_j^\neg)$ varies on a line for fixed $\Theta_j^\neg$, the total number of $Y$ is bounded by $10N_\star e^{10\rho N_1}$. Thus by \eqref{Apr6}, \eqref{Apr9} and $c_1< c_3/10$,
one has 
\begin{equation*}\label{Apr12}
\mathrm{mes}({\mathbf{X}}_{{N_\star}}(\Theta_j^\neg))\leq C(2N_\star+1)^{b}e^{10\rho N_1}e^{-\widetilde N^{1/3}} \leq e^{-{N_\star}^{c_3/7}}\leq e^{-{N_\star}^{c_1}}.
\end{equation*}

Suppose  now $\Theta\notin \mathbf{X}_{N_\star}$. Applying 
 Lemma \ref{res1} yields 
\begin{align*}
\|G_{Q}(E;\Theta)\|\leq 4(2 N^{c_4}+1)^{b}e^{\sqrt{N^{c_4}}}\leq e^{\sqrt{N_\star}}.
\end{align*}
Recalling  Lemma \ref{res2}, we obtain
\begin{align*}
  |G_{Q}(E;\Theta)(\mathbf{n},\mathbf{n}')|\leq  e^{- (\frac{4\rho}{5}-\frac{C(\rho,b,d)}{\sqrt{N_1}})|\mathbf{n}-\mathbf{n}'|}\  {\mathrm{for} \ |\mathbf{n}-\mathbf{n}'|\geq {N_\star}/{10}}.
\end{align*}
\end{proof}

{\bf STEP 3: The General Step}

The proof is  based on  similar arguments as in \textbf{STEP 2}.

We define for  $N\geq e^{|\log\varepsilon|^{1/4}}$ the following scales
\begin{align*}
N_1\sim (\log N)^{2/{c_1}},\ N_2\sim N_1^{2/{c_1}}.
\end{align*}
Then we have
\begin{prop}\label{claim}
Let $\Omega_{N_i}$ $(i=1,2)$ be semi-algebraic set satisfying $\deg(\Omega_{N_i})\leq N_i^{4d}$ and let $\bar\rho_i\in [\rho/2,\rho)$. Assume further the following holds:
If  $\vec\omega\in\Omega_{N_i}$ and $E\in\mathcal{I}_\varepsilon$, then there exists some semi-algebraic set $\mathbf{X}_{N_i}\subset \mathbb{R}^d$ satisfying $\deg(\mathbf{X}_{N_i})\leq N_i^{C}$  such that
 \begin{align*}\sup_{1\leq j\leq d,\Theta_j^\neg\in \R^{d-1}}\mathrm{mes}( \mathbf{X}_{N_i}(\Theta_j^\neg))\leq e^{-N_i^{c_1}},\end{align*}
and for  $\Theta\notin \mathbf{X}_{N_i}$, $Q\in\mathcal{E}_{N_i}^{0}$,
\begin{align*}
\|G_{Q}(E;\Theta)\|&\leq e^{\sqrt{N_i}},\\
|G_{Q}(E;\Theta)(\mathbf{n},\mathbf{n}')|&\leq  e^{-\bar\rho_i|\mathbf{n}-\mathbf{n}'|}\  {\mathrm{for} \ |\mathbf{n}-\mathbf{n}'|\geq {N_i}/{10}},\\
\nonumber&(i=1,2).
\end{align*}

Then exists some semi-algebraic set $\Omega_{{N}} \subset \Omega_{{N}_1}\cap\Omega_{{N}_2}$ with  $\deg(\Omega_{N})\leq {N}^{4d}$ and
$\mathrm{mes}((\Omega_{{N}_1}\cap\Omega_{{N}_2})\backslash\Omega_{{N}})\leq {N}^{-c_2}$
such that, if $\vec\omega\in\Omega_{N}$, then for $N\geq N_0(c_1,\rho,b,d)>0$
\begin{itemize}
\item[(i)] For all $E\in \mathcal{I}_\varepsilon$ and $\Theta\in \mathbf{B}(100bN^2)$, there is  $\frac{{N}^{c_3}}{10}<M<10{N}^{c_4}$ such that
for all  $\mathbf{k}\in  [-M,M]^{b}\setminus[-M^{\frac{1}{10b}},M^{\frac{1}{10b}}]^{b}$,  $ \Theta+\mathbf{k}\vec\omega \notin \mathbf{X}_{N_1}$.

\item[(ii)] Fix $N_\star\in[N,N^2]$. Then there exists some  set $\mathbf{X}_{N_\star}=\mathbf{X}_{N_\star}(E,\omega)\subset \mathbb{R}^d$  such that
 \begin{align*}\sup_{1\leq j\leq d,\Theta_j^\neg\in \R^{d-1}}\mathrm{mes}( \mathbf{X}_{N_\star}(\Theta_j^\neg))\leq e^{-N_\star^{c_1}},\end{align*}
and for $\Theta\notin \mathbf{X}_{N_\star}$, $Q\in\mathcal{E}_{N_\star}^{0}$,
\begin{align*}
 \|G_{Q}(E;\Theta)\|&\leq e^{\sqrt{N_\star}},\\
 |G_{Q}(E;\Theta)(\mathbf{n},\mathbf{n}')|&\leq  e^{-(\overline{\rho}_1-\frac{C(\rho,b,d)}{\sqrt{N_1}})|\mathbf{n}-\mathbf{n}'|}\  {\mathrm{for} \ |\mathbf{n}-\mathbf{n}'|\geq {N_\star}/{10}}.
\end{align*}
\end{itemize}
\end{prop}

\begin{proof}
The proof is similar to that in \textbf{STEP 2}, and we omit the details here.
\end{proof}

From the above arguments, we can propagate LDT from (scales) $$\left[N, N^{2/{c_1}}\right]$$ to $$\left[e^{N^{c_1/2}}, e^{2N^{c_1/2}}\right].$$  Thus a standard MSA induction (on scales) can  establish LDT on  the whole interval $[N_0,\infty]$ (see \cite{BGS02} or \cite{JLS20} for details).  This finishes the proof of Theorem \ref{ldt}.

\end{proof}

\section{Proof of Theorem \ref{mthm1}}
In this section we will prove Theorem \ref{mthm1} by using LDT and Aubry duality. 

We begin with a useful lemma:

\begin{lem}\label{shnol}
Let $\Psi\in H^2(\R^d)$ satisfy 
\begin{align}\label{sche}
(-\Delta+\varepsilon V(\vec\theta+x\vec\omega))\Psi(x)=E\Psi(x).
\end{align}
Let $\widehat{\Psi}$ be the Fourier transformation of $\Psi$. Then for a.e. $\Theta\in\R^d$  and any $\vec\theta\in\T^{b}$ the following holds: Let  $Z=\{Z_\mathbf{k}\}_{\mathbf{k}\in\Z^{b}}$ satisfy $Z_\mathbf{k}=e^{-\mathbf{i}\mathbf{k}\cdot\vec\theta}\widehat{\Psi}(\Theta+\mathbf{k}\vec\omega)$. Then
\begin{align}\label{dsch}
h(\Theta)Z=EZ,
\end{align}
where $h(\Theta)$ is given by \eqref{aubry}.
Moreover,  There is $C=C(\Theta,\Psi,b,d)\in (0,\infty)$ such that
\begin{align}\label{poly}
|Z_\mathbf{k}|\leq C(1+|\mathbf{k}|)^{5b}\ {\rm for}\ \forall\ \mathbf{k}\in\Z^{b}.
\end{align}
\end{lem}

\begin{proof}
It suffices to consider $\Psi\in C_0^\infty(\R^d)$. We note first \eqref{sche} is equivalent to
\begin{align}\label{fsch}
{\|}\xi{\|}^2\widehat{\Psi}(\xi)+\widehat{\varepsilon V\Psi}(\xi)=E\widehat{\Psi}(\xi),
\end{align}
where $\|\xi\|^2={\sum\limits_{i=1}^d|\xi_i|^2}$. Since $V$ is quasi-periodic (and analytic), we have $${V}(\vec\theta+{x\vec\omega})=\sum_{\mathbf{k}\in\Z^{b}}\widehat{V}_\mathbf{k}e^{\mathbf{i}\mathbf{k}\cdot(\vec\theta+{x\vec\omega})},$$
and as a result,
\begin{align}
\nonumber\widehat{ V\Psi}(\xi)&=\int_{\R^d}V(\vec\theta+{x\vec\omega})\Psi(x)e^{-\mathbf{i}x\cdot\xi}{\rm d} x\\
\nonumber&=\sum_{\mathbf{m}\in\Z^{b}}\int_{\R^d}\widehat{V}_\mathbf{m}e^{\mathbf{i}\mathbf{m}\cdot(\vec\theta+{x\vec\omega})}\Psi( x)e^{-\mathbf{i} x\cdot\xi}{\rm d}x\\
\nonumber&=\sum_{\mathbf{m}\in\Z^{b}}\widehat{V}_\mathbf{m}e^{\mathbf{i}\mathbf{m}\cdot\vec\theta} \int_{\R^d}\Psi(x)e^{-\mathbf{i} x\cdot(\xi-\mathbf{m}\vec\omega)}{\rm d} x\\
\label{vpsi}&=\sum_{\mathbf{m}\in\Z^{b}}\widehat{V}_\mathbf{m}e^{\mathbf{i}\mathbf{m}\cdot\vec\theta}\widehat{\Psi}(\xi-\mathbf{m}\vec\omega).
\end{align}
Combining \eqref{fsch} and \eqref{vpsi} yields
\begin{align}\label{fsch1}
{\|}\xi{\|}^2\widehat{\Psi}(\xi)+\varepsilon \sum_{\mathbf{m}\in\Z^{b}}\widehat{V}_me^{\mathbf{i}\mathbf{m}\cdot\vec\theta}\widehat{\Psi}(\xi-\mathbf{m}\vec\omega)=E\widehat{\Psi}(\xi).
\end{align}
Given $\mathbf{k}\in \Z^{b}$, we set  $\xi=\Theta+\mathbf{k}\vec\omega$ in \eqref{fsch1}. Then
\begin{align*}
&e^{-i\mathbf{k}\cdot\vec\theta}{\|}\Theta+\mathbf{k}\vec\omega{\|}^2\widehat{\Psi}(\Theta+\mathbf{k}\vec\omega)+\varepsilon \sum_{\mathbf{m}\in\Z^{b}}\widehat{V}_{\mathbf{k}-\mathbf{m}}e^{-\mathbf{i}\mathbf{m}\cdot\vec\theta}\widehat{\Psi}(\Theta+\mathbf{m}\vec\omega)\\
&\ \ =Ee^{-i\mathbf{k}\cdot\vec\theta}\widehat{\Psi}(\Theta+\mathbf{k}\vec\omega),
\end{align*}
which implies \eqref{dsch}.

We then deal with the polynomial bound \eqref{poly}. Since $\widehat{\Psi}\in L^2(\R^d)$, we have
\begin{align*}
\int_{\R^d} \sum_{\mathbf{k}\in\Z^{b}}\frac{|Z_\mathbf{k}|^2(\Theta)}{(1+|\mathbf{k}|)^{10b}}{\rm d}\Theta
&=\sum_{\mathbf{k}\in\Z^{b}}\frac{1}{(1+|\mathbf{k}|)^{10b}}\int_{\R^d} {|\widehat{\Psi}(\Theta+\mathbf{k}\vec\omega)|^2}{\rm d}\Theta\\
&=\|\widehat\Psi\|_{L^2}^2\sum_{\mathbf{k}\in\Z^{b}}\frac{1}{(1+|\mathbf{k}|)^{10b}}<\infty,
\end{align*}
which implies for a.e. $\Theta\in\R^d$,
\begin{align*}
\sum_{\mathbf{k}\in\Z^{b}}\frac{|Z_\mathbf{k}|^2(\Theta)}{(1+|\mathbf{k}|)^{10b}}:=C(\Theta,\Psi,b,d)<\infty.
\end{align*}
This means that for a.e. $\Theta\in\R^d$,
\begin{align*}
|Z_\mathbf{k}|\leq C(\Theta,\Psi,b,d)(1+|\mathbf{k}|)^{5b}\  {\rm for} \ \forall\  \mathbf{k}\in\Z^{b}.
\end{align*}

\end{proof}

We then prove our main result Theorem \ref{mthm1}

\begin{proof}[\bf Proof of Theorem \ref{mthm1}]

Let $$\widetilde \Theta=\frac{\Theta}{K}, \widetilde E=\frac{E}{K^2}, \varepsilon=\frac{\lambda}{K^2}.$$

Suppose now $H(\vec\theta)$ admits some eigenvalue $E\in\left[-K^2(\log\frac{K^2}{\lambda})^{\frac{1}{2c_1}}, K^2(\log\frac{K^2}{\lambda})^{\frac{1}{2c_1}}\right]$, i.e., there is some $\Psi\in H^2(\R^d)$ such that  $H(\vec\theta)\Psi=E\Psi, \Psi\neq0$. Then we have $\widetilde E\in \mathcal{I}_\varepsilon$ (see \eqref{ebd}) and
\begin{align*}
(-\Delta+\lambda V(\vec\theta+Kx\vec\omega))\Psi(x)=E\Psi(x).
\end{align*}
Using Lemma \ref{shnol}, we know $Z(\Theta)=\{Z_\mathbf{k}=e^{-\mathbf{k}\cdot\vec\theta}\widehat\Psi(\Theta+K\mathbf{k}\vec\omega)\}_{\mathbf{k}\in\Z^{b}}$ satisfies
\begin{align}\label{latt}
\sum_{\mathbf{k}'\in\Z^{b}} \lambda\widehat V_{\mathbf{k}-\mathbf{k}'}Z_{\mathbf{k}'}(\Theta)+\sum_{i=1}^{d}(\Theta_i+K\mathbf{k}_i\cdot\vec\omega_i)^2Z_{\mathbf{k}}(\Theta)=EZ_{\mathbf{k}}(\Theta).
\end{align}
Obviously, \eqref{latt} is equivalent to
\begin{align*}
\sum_{\mathbf{k}'\in\Z^{b}} \frac{\lambda}{K^2}\widehat V_{\mathbf{k}-\mathbf{k}'}Z_{\mathbf{k}'}(\Theta)+\sum_{i=1}^{d}({\Theta_i}/{K}+\mathbf{k}_i\cdot\vec\omega_i)^2Z_{\mathbf{k}}(\Theta)
=\frac{E}{K^2}Z_{\mathbf{k}}(\Theta),
\end{align*}
or in the new coordinate $(\widetilde\Theta,\widetilde E,\varepsilon)$,
\begin{align}\label{dvk}
\sum_{\mathbf{k}'\in\Z^{b}}
\varepsilon\widehat V_{\mathbf{k}-\mathbf{k}'}\widetilde Z_{\mathbf{k}'}(\widetilde\Theta)+\sum_{i=1}^{d}(\widetilde \Theta_i+\mathbf{k}_i\cdot\vec\omega_i)^2\widetilde Z_{\mathbf{k}}(\widetilde\Theta)=\widetilde E \widetilde Z_{\mathbf{k}}(\widetilde\Theta),
\end{align}
where
\begin{align}\label{tzd}
\widetilde Z_\mathbf{k}(\widetilde\Theta):=Z_\mathbf{k}(\Theta)\ {\rm for}\ \forall\  \mathbf{k}\in\Z^{b}.
\end{align}

Note that \eqref{dvk} has the form of $h(\widetilde\Theta)\widetilde Z=\widetilde E\widetilde Z$ with $h(\widetilde\Theta)$ being given by \eqref{aubry}. Furthermore, recalling  \eqref{poly} and \eqref{tzd}, we have for a.e. $\widetilde\Theta$,
\begin{align}\label{zkt}
|\widetilde Z_\mathbf{k}(\widetilde\Theta)|\leq C(\widetilde\Theta,\Psi,b,d)(1+|\mathbf{k}|)^{5b}.
\end{align}


Let $\delta>0$ be given. We obtain by Theorem \ref{ldt} that for $0<\varepsilon=\frac{\lambda}{K^2}\leq \varepsilon_0(\delta,c_1,\rho,b,d),$
\begin{align*}
{\rm mes}\left([0,2\pi]^{b}\setminus \bigcap_{N\geq N_0}\Omega_N\right)\leq \delta.
\end{align*}
We  fix $\vec\omega\in \bigcap\limits_{N\geq N_0}\Omega_N$. Then since $\widetilde E\in \mathcal{I}_{\varepsilon}$ and $N\geq N_0$,  there exists some  set $\mathbf{X}_N=\mathbf{X}_N(\vec\omega,\widetilde E)\subset \mathbb{R}^d$  such that
 \begin{align}\label{xnb}
 \sup_{1\leq j\leq d,\widetilde \Theta_j^\neg\in \R^{d-1}}\mathrm{mes}(\mathbf{X}_N(\widetilde \Theta_j^\neg))\leq e^{-N^{c_1}},
 \end{align}
and for  $\widetilde\Theta\notin \mathbf{X}_N$, $Q\in\mathcal{E}_N^{0}$,
\begin{align}
\label{gf1} \|G_{Q}(\widetilde E;\widetilde\Theta)\|&\leq e^{\sqrt{N}},\\
\label{gf2}|G_{Q}(\widetilde E;\widetilde\Theta)(\mathbf{n},\mathbf{n}')|&\leq  e^{-\frac{\rho}{10}|\mathbf{n}-\mathbf{n}'|}\  {\mathrm{for} \ |\mathbf{n}-\mathbf{n}'|\geq {N}/{10}}.
\end{align}
We should remark that \eqref{gf1} and \eqref{gf2} obviously hold true if $|\widetilde\Theta|\geq 10bN$. Without loss of generality, we may assume $\mathbf{X}_N\subset \mathbf{B}(10bN)$. As a result, we have by Fubini Theorem and \eqref{xnb} that
\begin{align*}
{\rm mes}(\mathbf{X}_N)\leq CN^{d-1}e^{-N^{c_1}}\leq e^{-N^{c_1}/2}
\end{align*}
and
\begin{align*}
\sum_{N\geq N_0}{\rm mes}(\mathbf{X}_N)<\infty.
\end{align*}
Then we obtain by using Borel-Cantelli Theorem that
\begin{align*}
{\rm mes}(\mathbf{X}_\infty)=0,
\end{align*}
where  $$\mathbf{X}_\infty=\bigcap_{N\geq N_0}\bigcup_{M\geq N}\mathbf{X}_M.$$
Assuming $\widetilde\Theta\notin \mathbf{X}_\infty$, then there exists $N_1(\widetilde\Theta)\geq N_0$ such that $\widetilde\Theta\not\in \mathbf{X}_N$ for all $N\geq N_1$. Recall that the Poisson's identity: For $h(\widetilde\Theta)\widetilde Z=\widetilde E \widetilde Z$ and $\mathbf{n}\in \Lambda\subset\Z^{b}$,
\begin{align*}
\widetilde Z_\mathbf{n}(\widetilde\Theta)=-\varepsilon\sum_{\mathbf{n}'\in\Lambda,\mathbf{n}''\notin\Lambda}G_{\Lambda}(\widetilde E;\widetilde \Theta)(\mathbf{n},\mathbf{n}')\widehat{V}_{\mathbf{n}'-\mathbf{n}''}\widetilde Z_{\mathbf{n}''}(\widetilde\Theta).
\end{align*}
Thus  for a.e. $\widetilde\Theta$ and  $N\geq N_1$,
\begin{align}
\nonumber|\widetilde Z_{\mathbf{0}}(\widetilde\Theta)|&=|\sum_{|{\mathbf{n}}|\leq N,|{\mathbf{n}}'|>N}G_{[-N,N]^{b}}(\widetilde E;\widetilde \Theta)(\mathbf{0},{\mathbf{n}})\widehat{V}_{{\mathbf{n}}-{\mathbf{n}}'} \widetilde{Z}_{{\mathbf{n}}'}(\widetilde \Theta)|\\
\nonumber&\leq C(\widetilde\Theta,\Psi,b,d) \sum_{|\mathbf{n}|\leq N,|\mathbf{n}'|>N}e^{-\frac{\rho}{10}|{\mathbf{n}}|+\frac{\rho N}{100}+\sqrt{N}}e^{-\rho|{\mathbf{n}}-{\mathbf{n}}'|} |{\mathbf{n}}'|^{5b}\\
\nonumber&\leq C(\widetilde\Theta,\Psi, b,d)N^{b}\sum_{|{\mathbf{n}}'|>N}e^{-\frac{\rho}{10}|{\mathbf{n}}'|+\frac{\rho N}{100}+\sqrt{N}}|{\mathbf{n}}'|^{5b}\\
\label{psi0}&\leq e^{-\frac{\rho N}{20}}.
\end{align}
Letting $N\to\infty$ in \eqref{psi0}, we have $\widetilde Z_{\mathbf{0}}(\widetilde\Theta)=0$ for a.e. $\widetilde\Theta$. This implies
\begin{align*}
0=\int_{\R^d}|\widetilde Z_\mathbf{0}(\widetilde\Theta)|^2{\rm d}\widetilde \Theta=K^{-d}\int_{\R^d}|\widehat \Psi(\Theta)|^2{\rm d}\Theta,
\end{align*}
 and $\Psi=0$, a contradiction.

We have shown $H(\vec\theta)$ has no eigenvalues in $\left[-K^2(\log\frac{K^2}{\lambda})^{\frac{1}{2c_1}}, K^2(\log\frac{K^2}{\lambda})^{\frac{1}{2c_1}}\right]$.

\end{proof}



\appendix
\section{}
\begin{lem}\label{spt}
Let $H=-\Delta+V(x)$ be a Schr\"odinger operator with $M=\sup\limits_{x\in \R^d}|V(x)|<\infty$. Then for any $E\geq 0$, we have
\begin{align*}
\sigma(H)\cap[E-M,E+M]\neq \emptyset,
\end{align*}
where $\sigma(H)$ denotes the spectrum of $H.$
\end{lem}

\begin{proof}
Let $\mathbb{E}(\cdot)$ be the (projection-valued) spectral measure of $H=-\Delta+V(x)$. Then the Spectral Theorem reads as
\begin{align*}
H=\int_{\sigma(H)}\lambda {\rm d}\mathbb{E}(\lambda).
\end{align*}
Note that $\sigma(-\Delta)=\sigma_{ess}(-\Delta)=[0,\infty)$. From Weyl criterion, for each $E\geq0$ there is a sequence $\{F_n(x)\}_{n\in\N}\subset H^2(\R^d)$ such that
\begin{align*}
&(-\Delta-E)F_n\to 0\ {\rm as}\ n\to\infty,\\
&\|F_n\|_{L^2}=1\ {\rm for}\ \forall\ n.
\end{align*}
For  $n\in \N$, we know $\mu_n(\cdot)=\langle \mathbb{E}(\cdot)F_n,F_n\rangle$ is a positive Borel measure and $\mu_n(\sigma(H))=\|F_n\|_{L^2}^2=1$,
where $\langle F,G\rangle=\int_{\R^d}F(x)\overline{G}(x){\rm d} x$. Obviously, $\|VF_n\|_{L^2}\leq M$. Hence
\begin{align}
\nonumber\inf_{\lambda\in\sigma(H)}|\lambda-E|^2\|F_n\|^2_{L^2}&\leq\int_{\sigma(H)}|\lambda-E|^2{\rm d}\mu_n(\lambda)\\
\nonumber&=\|(-\Delta+V-E)F_n\|_{L^2}^2\\
\label{fnm}&\leq (\|(-\Delta-E)F_n\|_{L^2}+M)^2.
\end{align}
Letting $n\to\infty$ in \eqref{fnm}, we get
\begin{align}\label{disb}
{\rm dist}(E, \sigma(H))\leq M,
\end{align}
which shows $[E-M,E+M]\cap \sigma(H)\neq\emptyset$. For otherwise, there must be ${\rm dist}(E, \sigma(H))\geq {2}M$, which contradicts with \eqref{disb}.
\end{proof}

\section{}

 We write $G_{(\cdot)}=G_{(\cdot)}(E;\Theta)$ for simplicity. 

 We have the following lemmas:

\begin{lem}[Lemma A.1, \cite{Shi19}]\label{pag}
Fix $\bar\rho>0$. Let $\Lambda\subset\mathbb{Z}^{b}$ satisfy $\Lambda\in\mathcal{E}_N$ and let $A,B$ be two linear operators  on $\C^\Lambda$.  We assume further
 \begin{align*}
\|A^{-1}\|&\leq e^{\sqrt{N}},\\
 |A^{-1}(\mathbf{n},\mathbf{n}')|&\leq e^{-\bar\rho|\mathbf{n}-\mathbf{n}'|}\ \mathrm{for}\  |\mathbf{n}-\mathbf{n}'|\geq N/10.
 \end{align*}
Suppose that for all $\mathbf{n},\mathbf{n}'\in\Lambda$, $$|(B-A)(\mathbf{n},\mathbf{n}')|\leq e^{-3\bar\rho N-\bar\rho|\mathbf{n}-\mathbf{n}'|}.$$
Then
\begin{align*}
\|B^{-1}\|&\leq 2\|A^{-1}\|,\\
 |B^{-1}(\mathbf{n},\mathbf{n}')|&\leq |A^{-1}(\mathbf{n},\mathbf{n}')|+e^{-\bar\rho|\mathbf{n}-\mathbf{n}'|}.
 \end{align*}
\end{lem}

\begin{lem}[Lemma 3.2, \cite{JLS20}]\label{res1}
 Let $ \bar\rho\in (\epsilon,\rho]$, $M_1\leq N$ and ${\rm diam}(\Lambda)\leq 2N+1$. Suppose that for any $\mathbf{n}\in \Lambda $, there exists some  $ W=W(\mathbf{n})\in \mathcal{E}_M$ with
$M_0\leq M\leq M_1$ such that
$\mathbf{n}\in W\subset \Lambda$,  ${\rm dist} (\mathbf{n},\Lambda \backslash W)\geq {M}/{2}$ and
\begin{align*}
\|G_{W}\|&\leq2 e^{\sqrt{M}},\\
|G_{W}(\mathbf{n},\mathbf{n}')|&\leq  2e^{-\bar\rho|n-n'|}\  {\mathrm{for} \ |\mathbf{n}-\mathbf{n}'|\geq {M}/{10}}.
\end{align*} We assume further that $M_0\geq M_0(\epsilon, \rho, b,d)>0$.
Then
\begin{align*}
  \|G_{\Lambda}\|\leq 4 (2M_1+1)^{b} e^{\sqrt{M_1}}.
\end{align*}
\end{lem}

\begin{lem}[Theorem 3.3, \cite{JLS20}]\label{res2}
Let $\Lambda_1\subset\Lambda\subset\Z^{b}$ satisfy ${\rm diam}(\Lambda)\leq 2N+1$, $ {\rm diam}(\Lambda_1)\leq N^{\frac{1}{3b}}$. Let $ M_0\geq (\log N)^{2}$ and $\bar\rho\in \left[\frac{\rho}{2},\rho\right]$.
Suppose that for any $\mathbf{n}\in \Lambda \backslash\Lambda_1$, there exists some  $ W=W(\mathbf{n})\in \mathcal{E}_M$ with
$M_0\leq M\leq N^{1/3}$ such that
$\mathbf{n}\in W\subset \Lambda\backslash\Lambda_1$, ${\rm dist} (\mathbf{n},\Lambda\backslash \Lambda_1\backslash W)\geq {M}/{2}$  and
\begin{align*}
\|G_{W}\|&\leq e^{\sqrt{M}},\\
 |G_{W}(\mathbf{n},\mathbf{n}')|&\leq  e^{- \bar{\rho}|\mathbf{n}-\mathbf{n}'|}\  {\mathrm{for} \ |\mathbf{n}-\mathbf{n}'|\geq {M}/{10}}.
\end{align*}
Suppose further that
\begin{align*}
\|G_{\Lambda}\|\leq e^{\sqrt{N}}.
\end{align*}
Then
\begin{align*}
 |G_{\Lambda}(\mathbf{n},\mathbf{n}')|\leq e^{-(\bar{\rho}-\frac{C}{\sqrt{M_0}})|\mathbf{n}-\mathbf{n}'|}\ \mathrm{for}\  |\mathbf{n}-\mathbf{n}'|\geq{N}/{10},
\end{align*}
where $C=C(\rho,b,d)>0$ and $N\geq N_0(\rho,b,d)>0$.
\end{lem}


 \end{document}